\documentclass[11pt]{article}
\usepackage{amsmath, amssymb, amsfonts, amstext, amsthm, textcomp, enumerate}
\usepackage[mathscr]{euscript}
\usepackage{float}
\usepackage{booktabs}
\usepackage{mathtools}
\usepackage{graphicx}
\usepackage{caption}
\usepackage{epstopdf}
\usepackage{longtable}
\usepackage[utf8]{inputenc}
\usepackage{xcolor}
\usepackage{color}
\usepackage{hyperref}
\usepackage{graphicx}
\usepackage{dcolumn}% Align table columns on decimal point
\usepackage{bm}% bold math
\usepackage{epstopdf}
\usepackage[english]{babel}
\usepackage{subfigure}
\usepackage{xcolor}

\usepackage{ulem}
\newtheorem{thm}{Theorem}[section]
\newtheorem{lem}[thm]{Lemma}%[section]
\newtheorem{cor}[thm]{Corollary}%[section]
\newtheorem{pro}[thm]{Proposition}%[section]
\newtheorem{defn}[thm]{Definition}%[section]

\newtheorem{rem}[thm]{Remark}%[section]
\newcommand{\diag}{\operatorname{diag}}

\newcommand{\mnr}{\mathbf{M}_n\,(\mathbb{R})}

\newcommand{\IM}{\mbox{$\mathbf{IM}$}}
\newcommand{\IDM}{\mbox{$\mathbf{IDM}$}}
\newcommand{\SIDM}{\mbox{$\mathbf{SIDM}$}}

\newfont{\bb}{msbm10}

\newtheorem{theoremA}{Theorem}
\newtheorem{theoremB}{Theorem}

\begin{document}
\allowdisplaybreaks
	\title{Linear Preservers of Infinitely Divisible Matrices}
    \author{Shaun Fallat\thanks{Department of Mathematics and Statistics,
University of Regina, Regina, SK, Canada \\
(sfallat@uregina.ca).}
    \and  Samir Mondal\thanks{Department of Mathematics and Statistics,
University of Regina, Regina, SK, Canada 
(Samir.Mondal@uregina.ca).} }
    \date{\today}
 \maketitle

%\begin{abstract}
%We characterize bijective linear preservers on $M_n(\mathbb{R})$ associated with infinite divisibility of non-negative matrices. More precisely, we obtain complete characterizations of bijective linear maps preserving (i) non-negativity, (ii) strongly infinitely divisible matrices, and (iii) infinitely divisible matrices. Our proofs combine the Inverse Function Theorem with Zariski-density techniques, building on the recent approach of Fallat and Mondal [\textit{Proc. Amer. Math. Soc.}, 2026].
%{\color{red} We characterize bijective linear preservers on $M_n(\mathbb{R})$ associated with the notion of infinite divisibility of nonnegative matrices. Specifically, we provide complete classifications of the bijective linear maps that preserve (i) nonnegativity, (ii) strong infinite divisibility, and (iii) infinite divisibility. Our approach combines techniques from differential topology and algebraic geometry, most notably the Inverse Function Theorem and Zariski-density arguments. Building on the framework recently applied in Fallat and Mondal [\textit{Proc. Amer. Math. Soc.}, 2026], we show that the preservers of these matrix classes are precisely those arising from standard monomial similarities. }
%\end{abstract}

\begin{abstract}
An infinitely divisible nonnegative matrix is an entry-wise nonnegative matrix that admits an entry-wise nonnegative $m$th root, with respect to usual matrix multiplication, for every positive integer $m$; it is called strongly infinitely divisible when it is, in addition, invertible. The study of such matrices has its origins in the theory of infinitely divisible probability distributions and the embedding problem for Markov matrices and is closely connected with continuous one-parameter semigroups; in the invertible case, this connection admits a natural description in terms of the exponential map. In this paper, we characterize the bijective linear maps on $M_n(\mathbb{R})$ that preserve strongly infinitely divisible matrices and infinitely divisible nonnegative matrices. In the former case, we combine the Inverse Function Theorem with Zariski-density techniques, using the Zariski-density approach recently developed in linear preserver theory by Fallat and Mondal [\textit{Proc. Amer. Math. Soc.}, 2026]. 
In the latter case, we first show that preserving infinite divisibility
forces preservation of the cone of entry-wise nonnegative matrices and
then exploit the additional structure afforded by infinite divisibility
to complete the characterization.

% (X) In the former case, we combine the Inverse Function Theorem with Zariski-density techniques, following a recent approach developed in linear preserver theory by Fallat and Mondal [\textit{Proc. Amer. Math. Soc.}, 2026], and in the latter case we focus on additional structure afforded to infinitely divisible matrices beyond entry-wise nonnegativity.

% (X) In the former case, we combine the Inverse Function Theorem with Zariski-density techniques, using a recent Zariski-density approach developed in linear preserver theory by Fallat and Mondal [\textit{Proc. Amer. Math. Soc.}, 2026]. In the latter case, we exploit the additional structure afforded by infinite divisibility beyond entry-wise nonnegativity.

\end{abstract}

% edit keywords and MSC as needed

\noindent AMS-MSC: 15A86, 15A16 (primary); 54C05, 15B48 (secondary)

\section{Introduction and Main Results}
Linear preserver problems form a classical and active research area of matrix and operator theory concerned with determining the general form of a linear operator on spaces of matrices or bounded linear operators that leave certain functions, subsets, relations, or other structures invariant. Typical examples include transformations that preserve rank, determinant, spectrum, invertibility, or various forms of positivity; see, for example, \cite{GutermanLiSemrl2000,LiPierce2001,Molnar2007} for general accounts of the subject. The theory goes back to Frobenius's study of determinant-preserving transformations \cite{Frobenius1897}, and subsequent work has addressed preservers of rank and related matrix properties \cite{MarcusMoyls1959}, the unitary group \cite{Marcus1959}, the spectrum \cite{JafarianSourour1986} and invertibility \cite{Sourour1996}. Preserver problems involving positivity and related matrix classes have also received considerable attention, including certain positivity classes \cite{BermanHershkowitzJohnson1985}, the copositive cone \cite{Shitov2021}, the completely positive rank \cite{Shitov2023}, and sign regularity \cite{ChoudhuryYadav2025SignRegularity}. More recently, Fallat and Mondal \cite{FallatMondal2026} employed Zariski-density arguments in studying linear preservers of real matrix classes admitting a real logarithm, apparently the first such use in linear preserver theory. The Zariski-density techniques developed there also play a central role in our analysis.

A classical result of Beasley and Laffey~\cite{BeasleyLaffey1990} shows that bijective linear maps preserving invertibility are of the form
\begin{equation}\label{standardtransformations}
\varphi(A)=PAQ \qquad \text{or} \qquad \varphi(A)=PA^{t}Q,
\end{equation}
where $P$ and $Q$ are invertible. Such mappings are referred to as \textit{standard transformations} (or transformations of standard form)~\cite{LiPierce2001}. In the infinite-dimensional setting, Sourour~\cite{Sourour1996} obtained analogous structure results for bijective invertibility-preserving linear maps on algebras of bounded linear operators on Banach spaces.

To study the matrix class considered in this paper, we recall some basic facts concerning matrix roots and logarithms. For a positive integer $k$, a matrix $X\in M_n(\mathbb{C})$ is called a $k$th root of $A$ if $X^k=A$; see \cite{hig} for the general theory of matrix roots. Similarly, $B\in M_n(\mathbb{C})$ is a logarithm of $A$ if $A=e^B$. Every invertible complex matrix admits a logarithm \cite[Theorem 2.10]{hall}, and therefore every matrix in $\mathrm{GL}_n(\mathbb{C})$ is a matrix exponential.

The situation in the real field is more restrictive. We say that $A\in M_n(\mathbb{R})$ has \textit{all real roots} if, for every positive integer $k$, there exists $X\in M_n(\mathbb{R})$ such that $X^k=A$, and define
$\mathcal{G}
=
\left\{
A\in\mathrm{GL}_n(\mathbb{R}) :
A\text{ has all real roots}
\right\}.$
A fundamental characterization \cite{MST} shows that
$\mathcal{G}
=
K_n^*
=
\left\{
e^B : B\in M_n(\mathbb{R})
\right\}.
$ 

Thus, a real invertible matrix has roots of every positive integer order precisely when it admits a real logarithm. However, unlike the complex case, not every real invertible matrix has this property. For example,
$
\det(e^B)=e^{\operatorname{tr}(B)}>0,
$
so a real matrix with negative determinant cannot belong to $\mathcal{G}$. The precise obstruction to the existence of a real logarithm is described by the following classical theorem of Culver.

\begin{thm}[Culver, \cite{culver}]
\label{existreallog}
{\rm Let $A\in M_n(\mathbb{R})$. Then $A$ has a real logarithm, That is, there exists a real matrix $X$ such that $A=e^X$, if and only if the following conditions are satisfied:
\begin{enumerate}
    \item[(i)] $A$ is invertible.
    \item[(ii)] Each elementary divisor (Jordan block) of $A$ belonging to a negative eigenvalue occurs an even number of times.
\end{enumerate}}
\end{thm}

This characterization shows that $\mathcal{G}$ is a proper and structurally distinguished subset of $\mathrm{GL}_n(\mathbb{R})$. It is therefore natural to ask which linear transformations preserve this class. This question was recently resolved by Fallat and Mondal \cite{FallatMondal2026}, who obtained the following characterization.

\begin{thm}[Fallat and Mondal, \cite{FallatMondal2026}]
Let $\varphi:M_n(\mathbb{R})\to M_n(\mathbb{R})$ be a bijective linear map. Then
$\varphi(\mathcal{G})=\mathcal{G}$ 
if and only if there exist $P\in\mathrm{GL}_n(\mathbb{R})$ and $c>0$ such that
$\varphi(A)=cPAP^{-1}$ or $\varphi(A)=cPA^TP^{-1}$, 
for all $A\in M_n(\mathbb{R})$.
\end{thm}

 In this paper, we study linear operators that preserve infinite divisibility within the cone of non-negative matrices. The theory of infinite divisibility, originating in probability and Markov processes, is closely intertwined with semigroup theory, particularly through the embedding of matrices into continuous one-parameter semigroups. In this setting, Ott (see \cite{Ott1972, Ott1975}) showed that an infinitely divisible entry-wise nonnegative matrix can be characterized by its embedding into a continuous semigroup of nonnegative matrices. In the invertible case, Van-Brunt~\cite[Theorem 1]{van} obtained a further characterization in terms of the exponential map later.

Now we will focus on the main classes of matrices under consideration.  The nonnegativity of arrays in our discussion is meant to be entry-wise. 
			\begin{defn} {\rm
				A nonnegative matrix $A\in\mnr$ is said to be an {\it infinitely divisible matrix} 
				{($\IDM_n$, $A\in\IDM_n$)} 
				if  there exists a sequence of nonnegative matrices $\{K_m\}_{m=1}^{\infty}$ such that 
				$(K_m)^m=A.$
				In other words, $A$ is infinitely divisible if for every positive integer $m$,
				there exists an $m$th matrix root of $A$ that is a nonnegative matrix.
				If an infinitely divisible matrix $A$ is also invertible, then 
				we refer to $A$ as a {\it strongly infinitely divisible matrix}
				{($\SIDM_n$, $A\in\SIDM_n$)}.		
			}
		\end{defn}
We may leave off the subscript for the sets above when the dimension is evident from context.

In contrast, we note that the terminology \textit{infinitely
divisible matrix} is also used in the literature with respect to positive
semidefinite matrices; see, for example, Horn~\cite{Horn1969}. In
this setting, infinite divisibility refers to the preservation of
positive semidefiniteness under positive fractional entrywise
(Hadamard) powers. This differs from our notion of infinite
divisibility, which is defined in terms of nonnegative roots under
ordinary matrix multiplication.

% \begin{thm} {\rm \cite[Theorem 1]{van}}
% 			\label{SIDM}
% The matrix $A\in\mnr$ is \SIDM\ if and only if there exists an essentially nonnegative $B$ such that $A^t=e^{tB}$ for all $t\geq 0$. 
% 			%\sout{In particular, $A=e^{B}$. Moreover, $B$ can be chosen to be
% 			%the principal logarithm of $A$, $\log A$.}} {\bf I believe we don't need to include this %sentence here since we mention it in the next corollary.}
% 		\end{thm}
		
% Let $F$ be a family of $n\times n$ matrices. The connection between infinite divisibility and semigroup structures is captured by continuous semigroups. A \textit{continuous semigroup} in $F$ is a continuous mapping
% $[0,\infty)\to F$, $t\mapsto B_t$,
% satisfying $B_{t+s}=B_tB_s$ for all $s,t\ge0$. The following characterization due to Ott is fundamental.

% \begin{thm}[Ott~\cite{ott1975}, Theorem~1]
% A matrix $A$ is infinitely divisible if and only if there exists a continuous semigroup $\{B_t\}_{t\ge0}$ of nonnegative matrices such that $B_1=A$.
% \end{thm}

%\sout{The connection between infinite divisibility and semigroup structures
%is made precise by a fundamental result of Ott.} 
Infinite divisibility as defined here is closely connected with continuous one-parameter semigroups, a connection arising naturally in probability and Markov processes. This relationship is made more precise by a fundamental result of Ott. Recall that a
\textit{continuous semigroup} of nonnegative matrices is a continuous
family $\{B_t\}_{t\ge0}$ that satisfies
$B_{t+s}=B_tB_s$ for $s,t\ge0$.

\begin{thm}[Ott~\cite{Ott1975}, Theorem~1]
A nonnegative matrix $A$ is infinitely divisible if and only if there
exists a continuous semigroup $\{B_t\}_{t\ge0}$ of nonnegative
matrices such that $B_1=A$.
\end{thm}

In the invertible case, this semigroup representation leads to a
particularly useful characterization of strongly infinitely divisible
matrices. Van-Brunt~\cite{van} proved the following result; a new proof
based on Ott's theorem was later given in~\cite{MST}.

\begin{thm}[Van-Brunt~\cite{van}, Theorem~1]
\label{SIDM}
A matrix $A\in M_n(\mathbb R)$ is strongly infinitely divisible if
and only if there exists an essentially nonnegative matrix $B$ such
that $A=e^B$,
where essentially nonnegative means that all off-diagonal entries of
$B$ are nonnegative.
\end{thm}

A second motivation for studying these matrices comes from a problem
posed by C.~R.~Johnson~\cite{Johnson1982} concerning inverse $M$-matrices.
Recall that a nonnegative matrix is called an inverse $M$-matrix if
its inverse is a nonsingular $M$-matrix, and let $\IM$ denote the
class of such matrices. The structure and characterization of inverse
$M$-matrices have been studied extensively; see, for example,
\cite{FM,FHS,Johnson1982} and the references therein.

Johnson~\cite{Johnson1982} asked whether $\IM$ coincides with the class of
nonsingular nonnegative matrices that admit nonnegative roots of every
positive integer order. In the terminology used here, the question is
whether every strongly infinitely divisible matrix is an inverse
$M$-matrix. This question was answered in negative terms by
Van-Brunt~\cite{van}. More precisely, every inverse $M$-matrix is
strongly infinitely divisible, whereas the converse does not hold.
Consequently,
$
\IM\subsetneq\SIDM.
$
Interestingly, the situation is different in dimension two.
 in~\cite{MST} it was shown that for $2\times2$ matrices these  two classes
coincide.

\begin{thm}\cite{MST}\label{imsidm2}
{\rm Let $A\in M_2(\mathbb R)$. Then
$
A\in\IM$
if and only if
$A\in\SIDM.
$}
\end{thm}
% For nonnegative $2\times2$ matrices, membership in $\SIDM$ admits the
% following particularly simple characterization.
% \begin{thm}
% If $A$ is a $2\times 2$ nonnegative matrix, then $A\in \SIDM$ if and only if ${\rm det}(A)>0.$
% \end{thm}
Of particular relevance to the present work is the linear preserver
problem for inverse $M$-matrices, studied in
\cite{Johnson1982,TamLiou1990}. It was shown that such preservers are
precisely of the form
$\Phi(A)=D_1(P^TAP)D_2$
or $\Phi(A)=D_1(P^TA^TP)D_2$, for all $A$,
where $D_1,D_2$ are positive diagonal matrices and $P$ is a
permutation matrix.
% \begin{thm}[Linear preservers of inverse $M$-matrices {\cite{Johnson1982, TamLiou1990}}]
% A linear map $\Phi$ satisfies $\Phi(\mathcal{IM}_n)=\mathcal{IM}_n$ if and only if it is a composition of:
% \begin{enumerate}
%   \item Positive diagonal equivalence: $A\mapsto FAE$, where $E,F>0$ are diagonal;
%   \item Permutation similarity: $A\mapsto P^T A P$;
%   \item Possibly transposition: $A\mapsto A^T$.
% \end{enumerate}
% That is,
% $$
% \Phi(A) = F\, (P^T A^{\varepsilon} P)\, E, \qquad
% A^{\varepsilon}\in\{A,A^T\}.
% $$
% \end{thm}

Our main contribution is the classification of bijective linear maps
preserving infinitely divisible matrices. We begin with the invertible
class $\SIDM$. Using the Inverse Function Theorem, we show that
\SIDM~ has nonempty Euclidean interior; more precisely, it contains
a neighborhood of the identity matrix
(Lemma~\ref{lem:SIDM_open}). Combining this local property with a
Zariski-density argument, we prove that every bijective linear map
preserving \SIDM \ also preserves
$GL_n(\mathbb R)$
(Lemma~\ref{lem:SIDM_preserver_GL}). This leads to our first main
result.

% \medskip
% \noindent
% \begin{theoremA}[Linear preservers of \SIDM]\label{thm:mainA}    
% Let
% $\varphi:M_n(\mathbb R)\to M_n(\mathbb R)$ be a bijective linear map.
% Then $\varphi(\SIDM_n)=\SIDM_n$ if and only if one of the following holds:
% \begin{enumerate}
% \item If $n=1$, there exists a scalar $c>0$ such that
% \[
% \varphi(A)=cA
% \qquad\text{for all }A\in M_1(\mathbb R).
% \]

% \item If $n=2$, there exist a permutation matrix $P$ and positive
% diagonal matrices $D_1,D_2$ such that either
% \[
% \varphi(A)=D_1PAP^TD_2
% \qquad\text{or}\qquad
% \varphi(A)=D_1PA^TP^TD_2
% \]
% for all $A\in M_2(\mathbb R)$.

% \item If $n\ge3$, there exist a scalar $c>0$ and a positive monomial
% matrix $S$ such that either
% \[
% \varphi(A)=cSAS^{-1}
% \qquad\text{or}\qquad
% \varphi(A)=cSA^TS^{-1}
% \]
% for all $A\in M_n(\mathbb R)$.
% \end{enumerate}
% \end{theoremA}

\medskip
\noindent
\begin{theoremA}[Linear preservers of \SIDM]\label{thm:mainA}
{\rm Let $\varphi:M_n(\mathbb R)\to M_n(\mathbb R)$ be a bijective linear map.
Then $\varphi(\SIDM_n)=\SIDM_n$ if and only if one of the following holds:
\begin{enumerate}
\item If $n=1$, there exists a scalar $c>0$ such that
$\varphi(A)=cA$ for all $A\in M_1(\mathbb R)$.
\item If $n=2$, there exist a permutation matrix $P$ and positive
diagonal matrices $D_1,D_2$ such that either
$\varphi(A)=D_1PAP^TD_2$ or $\varphi(A)=D_1PA^TP^TD_2$
for all $A\in M_2(\mathbb R)$.
\item If $n\ge3$, there exist a scalar $c>0$ and a positive monomial
matrix $S$ such that either $\varphi(A)=cSAS^{-1}$ or
$\varphi(A)=cSA^TS^{-1}$ for all $A\in M_n(\mathbb R)$.
\end{enumerate}}
\end{theoremA}

We then turn to the full class \IDM, which also contains singular
infinitely divisible matrices. By determining the action of a
preserver on the nonnegative cone and subsequently analyzing the
induced structure on the matrix units and their coefficients, we
obtain our second main result.

\medskip
\noindent
\begin{theoremB}[Linear preservers of \IDM]\label{mainthm2}    
{\rm Let
$\varphi:M_n(\mathbb R)\to M_n(\mathbb R)$ be a bijective linear map.
Then $\varphi(\IDM_n)=\IDM_n$ if and only if one of the following holds:
\begin{enumerate}
\item If $n=1$, there exists a scalar $c>0$ such that
$
\varphi(A)=cA
\qquad\text{for all }A\in M_1(\mathbb R).
$

\item If $n=2$, there exist a permutation matrix $P$ and positive
diagonal matrices $D_1,D_2$ such that either
$\varphi(A)=D_1PAP^TD_2$ or $\varphi(A)=D_1PA^TP^TD_2$
for all $A\in M_2(\mathbb R)$.
\item If $n\ge3$, there exist a scalar $c>0$ and a positive monomial matrix $S$ such that either $\varphi(A)=cSAS^{-1}$ or $\varphi(A)=cSA^TS^{-1}$ for all $A\in M_n(\mathbb R)$.
\end{enumerate}}
\end{theoremB}

\section{Strongly Infinitely Divisible Matrices: Proof of Theorem A}

We begin our study of linear preservers with the class $\SIDM$. Our
goal in this section is to characterize all bijective linear maps on
$M_n(\mathbb R)$ that preserve $\SIDM.$ The proof
combines the topological structure of $\SIDM$ with classical results
on linear preservers of invertibility. In particular, we use the
inverse function theorem to establish a local property of $\SIDM$ and
then employ a Zariski-density argument to connect preservation of
$\SIDM$ with the preservation of invertibility. Together, these results
yield a complete characterization of the linear preservers of
$\SIDM$.

We first recall a useful observation concerning the behavior of bijective linear maps under closure.

\begin{lem}\cite{FallatMondal2026}
{\rm Let $L : M_n(\mathbb{R}) \to M_n(\mathbb{R})$ be a bijective linear map.  
Then the following are equivalent:
$
L(S) = S 
 \Longleftrightarrow 
L(\overline{S}) = \overline{S},
$ where $\overline{S}$ denotes the closure of the set $S$ in the Euclidean topology on $M_n(\mathbb{R})$.}
\end{lem}

The proof of Theorem~\hyperref[thm:mainA]{A}  requires several preliminary lemmas, beginning with an application of the inverse function theorem.

\subsection*{An Application of the Inverse Function Theorem}

We begin by recalling a general statement of the inverse function theorem on a real vector space; see, e.g., \cite[Chapter 9]{Rudin1976}

\begin{thm}[Inverse Function Theorem]\label{thm:IFT}
{\rm Let $f:\mathbb R^m\to\mathbb R^m$ be a continuously differentiable map, and let
$x_0\in\mathbb R^m$. If the derivative $Df(x_0)$ is invertible, then there exist
open neighborhoods $U$ of $x_0$ and $V$ of $f(x_0)$ such that
$f:U\to V$
is a bijection. Moreover, the inverse map $f^{-1}:V\to U$ is continuously
differentiable. In particular, $f(U)=V$ is open and $f$ is a local diffeomorphism
at $x_0$.}
\end{thm}

\medskip

Since $M_n(\mathbb R)$ is a finite-dimensional real vector space of dimension
$n^2$, it is naturally identified with $\mathbb R^{n^2}$, and hence
Theorem~\ref{thm:IFT} applies to smooth maps on $M_n(\mathbb R)$.

Before we come to our first lemma, we now show that \SIDM\ has nonempty interior. From Theorem \ref{SIDM} we have the characterization
\[
\SIDM=\{e^{-Q}: Q \text{ is a } Z\text{-matrix}\},
\]
where a $Z$-matrix is a real matrix whose off-diagonal entries are non-positive (that is $-Q$ is essentially nonnegative).
Define the set of strict $Z$-matrices
\[
\mathcal Z^\circ=\{Q\in M_n(\mathbb R): q_{ij}<0 \text{ for all } i\neq j\}.
\]
The set $\mathcal Z^\circ$ is open in $M_n(\mathbb R)$, since it is described by
finitely many strict inequalities.

\medskip
\begin{lem}\label{lem:SIDM_open}
{\rm Endowed with the usual Euclidean  topology, the set $\SIDM \subset M_n(\mathbb R)$
contains a nonempty open set. In particular, \SIDM \ contains a
neighborhood of the identity matrix $I$.}
\end{lem}
\begin{proof}
    Consider the matrix exponential map
\[
\exp : M_n(\mathbb R)\to M_n(\mathbb R), \qquad
\exp(X)=\sum_{k=0}^{\infty}\frac{X^k}{k!}.
\]
This map is smooth (indeed, real analytic) on $M_n(\mathbb R)$. Its derivative at
the zero matrix is given by
$D(\exp)_0(H)=H \qquad \text{for all } H\in M_n(\mathbb R), $ 
that is, $D(\exp)_0$ is the identity map on $M_n(\mathbb R)$ and hence is
invertible. By the Inverse Function Theorem, there exists $\delta>0$
such that $\exp$ restricts to a diffeomorphism from the open ball
\[
B(0,\delta)=\{X\in M_n(\mathbb R): \|X\|<\delta\}
\]
onto an open neighborhood of the identity matrix $I=\exp(0)$.

As noted above, the set $\mathcal Z^\circ$ is open in $M_n(\mathbb R)$. Let
$U:=\mathcal Z^\circ\cap B(0,\delta), $ 
where $\delta>0$ is chosen as above. Then $U$ is an open neighborhood of $0$
consisting entirely of $Z$-matrices. Since multiplication by $-1$ is a
homeomorphism of $M_n(\mathbb R)$, the map $Q\mapsto \exp(-Q)$ is also a local
diffeomorphism at $0$. Consequently,
$\exp(-U)$ 
is an open neighborhood of $I$ in $M_n(\mathbb R)$.
Finally, for every $Q\in U$, we have $Q$ a $Z$-matrix, and hence
$\exp(-Q)\in\SIDM$ by definition. Thus $\exp(-U)\subseteq \SIDM$, showing that
$I$ is an interior point of \SIDM\ and that \SIDM \ has nonempty interior in
the usual topology on $M_n(\mathbb R)$.
\end{proof}

% \begin{proof}
% We use the standard characterization of strongly infinitely divisible matrices:
% \[
% \SIDM_n=\{\exp(-Q): Q \text{ is a } Z\text{-matrix}\},
% \]
% where a $Z$-matrix is a real matrix whose off-diagonal entries are nonpositive.

% Define the set of {strict} $Z$-matrices
% \[
% \mathcal Z^\circ := \{Q\in M_n(\mathbb R): q_{ij}<0 \text{ for all } i\neq j\}.
% \]
% Since $\mathcal Z^\circ$ is defined by finitely many strict inequalities on the
% coordinates of $M_n(\mathbb R)\cong\mathbb R^{n^2}$, it is an open subset of
% $M_n(\mathbb R)$ in the usual topology.

% Consider the matrix exponential map
% \[
% \exp : M_n(\mathbb R)\to M_n(\mathbb R).
% \]
% This map is smooth, and its derivative at the zero matrix is the identity:
% \[
% D(\exp)_0(H)=H \qquad \text{for all } H\in M_n(\mathbb R).
% \]
% By the inverse function theorem, there exists $\delta>0$ such that $\exp$ restricts
% to a diffeomorphism from the open ball
% \[
% B(0,\delta):=\{Q\in M_n(\mathbb R): \|Q\|<\delta\}
% \]
% onto an open neighborhood of the identity matrix $I$.

% Now define
% \[
% U := \mathcal Z^\circ \cap B(0,\delta).
% \]
% Then $U$ is an open subset of $M_n(\mathbb R)$, and for every $Q\in U$,
% $Q$ is a $Z$-matrix. Hence
% \[
% \exp(-Q)\in \SIDM_n.
% \]
% Therefore the set
% \[
% \Omega := \{\exp(-Q): Q\in U\}
% \]
% is an open subset of $M_n(\mathbb R)$ (as the image of an open set under a local
% diffeomorphism), satisfies $\Omega\subseteq \SIDM_n$, and contains $I=\exp(0)$.

% This proves that \SIDM has nonempty interior in $M_n(\mathbb R)$.
% \end{proof}

\begin{rem}
{\rm The interior points of \SIDM\ obtained above are precisely those matrices
$B=e^{-Q}$ for which the generator $Q$ has strictly negative off-diagonal entries
and a sufficiently small norm.}
\end{rem}

%\begin{rem}
%The same argument works for any equivalent matrix norm on $M_n(\mathbb R)$,
%since all norms induce the same topology on finite-dimensional spaces.
%\end{rem}

% {\color{blue}
% We recall the following result, which will be used in the proof of the
% main lemma concerning bijective linear maps on \(\SIDM\).

% Before stating the result, we briefly recall the Zariski topology on
% \(M_n(\mathbb{R})\). A subset \(V\subseteq M_n(\mathbb{R})\) is said to be
% \textit{Zariski closed} if there exist polynomials
% \[
% f_1,\ldots,f_k\in\mathbb{R}[x_{11},\ldots,x_{nn}]
% \]
% such that
% \[
% V=\left\{A\in M_n(\mathbb{R}) :
% f_1(A)=\cdots=f_k(A)=0\right\}.
% \]
% A subset of \(M_n(\mathbb{R})\) is called \textit{Zariski open} if its
% complement is Zariski closed. The topology determined by these closed
% sets is called the \textit{Zariski topology} on \(M_n(\mathbb{R})\).

% For any subset \(S\subseteq M_n(\mathbb{R})\), its \textit{Zariski closure},
% denoted by \(\overline{S}^{\,Z}\), is the smallest Zariski-closed set
% containing \(S\). Equivalently,
% \[
% \overline{S}^{\,Z}
% =
% \bigcap_{\substack{V\text{ Zariski closed}\\ S\subseteq V}} V.
% \]
% If \(X\subseteq M_n(\mathbb{R})\) is equipped with the induced Zariski
% topology, then a subset \(S\subseteq X\) is said to be \textit{Zariski dense}
% in \(X\) if its closure in \(X\) is equal to \(X\).

% With this terminology, we have the following result.
% }

We recall the following notions and results, which are used in the proof
of the main lemma concerning bijective linear maps on \(\SIDM\). We consider the Zariski topology on \(\mnr\), which is particularly
well suited to the study of sets defined by polynomial equations. For a
detailed treatment, see \cite[Section~3.5]{reid1988}.
%\subsection*{Zariski Topology} 
A subset \(V\subseteq\mnr\) is called \textit{Zariski closed} if there
exist polynomials
$f_1,\ldots,f_k\in\mathbb{R}[x_{11},\ldots,x_{nn}]$ 
such that
$V=\{A\in\mnr:f_1(A)=\cdots=f_k(A)=0\}.$ 
The complement of a Zariski-closed set is called \textit{Zariski open}.
For a subset \(S\subseteq\mnr\), its \textit{Zariski closure}, denoted by
\(\overline{S}^{\,Z}\), is the smallest Zariski-closed set containing
\(S\). A subset \(S\) of a Zariski topological space \(X\) is said to be
\textit{Zariski dense} in \(X\) if its Zariski closure is \(X\).

We shall use the standard fact that a nonzero real polynomial cannot
vanish on a nonempty Euclidean-open subset of \(\mathbb{R}^m\). Indeed,
if a polynomial vanishes on a nonempty Euclidean-open set, then, by the
identity theorem for real-analytic functions, it must be the zero
polynomial. Consequently, if
\(V\subsetneq\mathrm{GL}_n(\mathbb{R})\) is a proper Zariski-closed
subset, then
$V=\mathrm{GL}_n(\mathbb{R})\cap Z(p_1,\ldots,p_k),$ 
for some polynomials \(p_1,\ldots,p_k\), at least one of which is nonzero.
Hence \(V\) cannot contain a nonempty relatively Euclidean-open subset of
\(\mathrm{GL}_n(\mathbb{R})\). Therefore, every nonempty relatively
Euclidean-open subset of \(\mathrm{GL}_n(\mathbb{R})\) is Zariski dense
in \(\mathrm{GL}_n(\mathbb{R})\).

We also recall the following elementary property of linear maps with
respect to the Zariski topology.

\begin{lem}\cite{FallatMondal2026}\label{lem:zariski_linear}
{\rm Let $\phi : M_n(\mathbb R) \to M_n(\mathbb R)$ be a linear map. Then $\phi$
is continuous with respect to the Zariski topology. In particular, if $\phi$
is bijective, then $\phi$ is a Zariski homeomorphism.}
\end{lem}

%\begin{lem}\label{lem:SIDM_zariski_dense}
%The Zariski closure of \SIDM\ in $M_n(\mathbb R)$ is equal to
%$GL_n(\mathbb R)$, that is,
%\[
%\overline{\SIDM}^{\,Z} = GL_n(\mathbb R).
%\]
%\end{lem}

\begin{lem}\label{lem:SIDM_zariski_dense}
{\rm With respect to the subspace Zariski topology on \(GL_n(\mathbb R)\)
induced by \(M_n(\mathbb R)\), the set \(\SIDM\) is Zariski dense; that is,
$
\overline{\SIDM}^{\,Z}=GL_n(\mathbb R).
$}
\end{lem}
\begin{proof}
By Lemma~\ref{lem:SIDM_open}, the set $\SIDM \subset M_n(\mathbb R)$ contains
a nonempty open subset with respect to the Euclidean topology. Let
$U \subset \SIDM$ be such an open set. In particular, $U \subset
GL_n(\mathbb R)$.

We claim $\overline{U}^{\,Z} = GL_n(\mathbb R)$. Suppose, for the sake of
contradiction, that there exists a proper Zariski-closed subset
$V \subsetneq GL_n(\mathbb R)$ such that $U \subseteq V$. Then there exists a
nonzero polynomial $p$ on $M_n(\mathbb R)$ that vanishes on $V$, and hence on
$U$. Since $U$ is Euclidean-open in $M_n(\mathbb R)$, the polynomial $p$ must vanish
on a nonempty open set. By analyticity of polynomials, this implies that $p$
vanishes identically on $M_n(\mathbb R)$, contradicting the fact that $p$ is
nonzero. Therefore, no proper Zariski-closed subset of $GL_n(\mathbb R)$ contains $U$.
Consequently,
$\overline{U}^{\,Z} = GL_n(\mathbb R).$ 

Since $U \subset \SIDM$, it follows that
\[
GL_n(\mathbb R)
= \overline{U}^{\,Z}
\subseteq \overline{\SIDM}^{\,Z}
\subseteq GL_n(\mathbb R),
\]
and hence $\overline{\SIDM}^{\,Z} = GL_n(\mathbb R),$ 
which completes the proof.
\end{proof}

Our final lemma is the key to characterizing bijective linear maps on \SIDM\ because it connects such maps to linear preservers of invertibility.

\begin{lem}\label{lem:SIDM_preserver_GL}
{\rm Let $\phi : M_n(\mathbb R) \to M_n(\mathbb R)$ be a bijective linear map such that
$\phi(\SIDM)=\SIDM.$ 
Then
$\phi\bigl(GL_n(\mathbb R)\bigr)=GL_n(\mathbb R).$ }
\end{lem}

\begin{proof}
Since $GL_n(\mathbb R)$ is a Zariski-open subset of $M_n(\mathbb R)$, the
Zariski topology on $GL_n(\mathbb R)$ coincides with the subspace topology
induced from $M_n(\mathbb R)$. By Lemma~\ref{lem:SIDM_zariski_dense}, the set
\SIDM \ is Zariski dense in $GL_n(\mathbb R)$, that is,
$\overline{\SIDM}^{\,Z} = GL_n(\mathbb R).$
Using Lemma~\ref{lem:zariski_linear} implies that every linear map on $M_n(\mathbb R)$ is
Zariski-continuous, and since $\phi$ is bijective, it is a Zariski
homeomorphism. Hence $\phi$ preserves Zariski closures, and we obtain
\[
\phi\!\left(\overline{\SIDM}^{\,Z}\right)
=\overline{\phi(\SIDM)}^{\,Z}
=\overline{\SIDM}^{\,Z}.
\]
Since $\overline{\SIDM}^{\,Z} = GL_n(\mathbb R)$, we conclude that
$\phi\bigl(GL_n(\mathbb R)\bigr)=GL_n(\mathbb R),$ 
as claimed.
\end{proof}

A final core ingredient for the proof of our main result concerns a formula for the logarithm of a triangular matrix.

\begin{thm}\cite[Theorem~4.11]{hig} [Function of a Triangular Matrix]
{\rm Let \(T \in \mathbb{C}^{n \times n}\) be upper triangular and suppose that \(f\) is defined on the spectrum of \(T\). Then \(F=f(T)\) is upper triangular with \(f_{ii}=f(t_{ii})\) and
\[
f_{ij}
=
\sum_{(s_0,\ldots,s_k)\in S_{ij}}
t_{s_0,s_1} t_{s_1,s_2} \cdots t_{s_{k-1},s_k}
\, f[\lambda_{s_0},\ldots,\lambda_{s_k}],
\]
where \(\lambda_i=t_{ii}\), \(S_{ij}\) is the set of all strictly increasing sequences of integers that start at \(i\) and end at \(j\), and \(f[\lambda_{s_0},\ldots,\lambda_{s_k}]\) is the \(k\)th order divided difference of \(f\) at \(\lambda_{s_0},\ldots,\lambda_{s_k}\).}
\end{thm}

Let $f(x)=\log x$. Using the standard formula for functions of upper triangular matrices via divided differences, we obtain the following.
\[
\bigl(\log(A(t)C)\bigr)_{13}
=
t^2
\left[
\frac{c_3}{2} f[c_1,c_3]
+
c_2c_3 f[c_1,c_2,c_3]
\right],
\]
where $f[\cdot,\cdot]$ and $f[\cdot,\cdot,\cdot]$ denote the first and second divided differences.
For completeness, we show that the coefficient $c_2$ is strictly negative.

\begin{rem}\label{rem:log-int-rep}
{\rm For \(f(x)=\log x\), the divided differences admit the integral representations
\[
f[a,c]
=
\int_0^\infty \frac{du}{(u+a)(u+c)},
\]
and
\[
f[a,b,c]
=
-\int_0^\infty
\frac{du}{(u+a)(u+b)(u+c)}.
\]

These identities provide convenient (and possibly well-known) integral representations for the divided
differences of \(\log x\), and will be useful in the proof of
Lemma~\ref{lem:log-dd}.}
\end{rem}

\begin{lem}\label{lem:log-dd}
{\rm Let \(f(x)=\log x\). If 
$
0<a\le c\le b,\;
a<b,
$
then
$
\frac{c}{2}f[a,c]+bc\,f[a,b,c]<0.
$}
\end{lem}

\begin{proof}
Using the integral representations in Remark~\ref{rem:log-int-rep},
\[
f[a,c]
=
\int_0^\infty \frac{du}{(u+a)(u+c)},
\qquad
f[a,b,c]
=
-\int_0^\infty
\frac{du}{(u+a)(u+b)(u+c)},
\]
we may write
\[
\frac{c}{2}f[a,c]+bc\,f[a,b,c]
=
\frac{c}{2}
\int_0^\infty
\frac{u-b}{(u+a)(u+b)(u+c)}\,du.
\]
Hence, since \(c>0\), it is enough to prove that
\[
I:=
\int_0^\infty
\frac{u-b}{(u+a)(u+b)(u+c)}\,du
<0.
\]

Split the integral at \(u=b\). Applying the change of variables
\(u=b^2/v\) on \((b,\infty)\), after simplification we obtain
\[
I
=
\int_0^b
(b-u)
\left[
\frac{b^3}
{(b^2+au)(b^2+bu)(b^2+cu)}
-
\frac1{(u+a)(u+b)(u+c)}
\right]du.
\]

Because \(b^2+bu=b(u+b)\), the bracketed term is negative precisely when
\[
b^2(u+a)(u+c)
<
(b^2+au)(b^2+cu).
\]
Observe that
$
(b^2+au)(b^2+cu)-b^2(u+a)(u+c)
=
(b^2-u^2)(b^2-ac).
$
Since \(0<u<b\) and \(ac<b^2\), the right-hand side is positive.
Hence the integrand is negative on \((0,b)\), and therefore \(I<0\). Consequently,
$
\frac{c}{2}f[a,c]+bc\,f[a,b,c]<0. 
$
\end{proof}

We are now in a position to prove our main result concerning a classification of the bijective maps on \SIDM.

% \begin{thm}[Linear preservers of \SIDM]\label{SIDMpreserver}
% Let $\varphi:M_n(\mathbb R)\to M_n(\mathbb R)$ be a bijective linear map.
% Then
% \[
% \varphi(\SIDM)=\SIDM
% \]
% if and only if there exist a monomial matrix $S$ and a scalar $c>0$ such that either
% \[
% \varphi(A)=c\,SAS^{-1}\quad\text{for all }A,
% \qquad\text{or}\qquad
% \varphi(A)=c\,SA^TS^{-1}\quad\text{for all }A.
% \]
% \end{thm}
% \begin{theoremA}[Linear preservers of \SIDM]
% %\label{thm:mainA}
% Let $\varphi:M_n(\mathbb R)\to M_n(\mathbb R)$ be a bijective linear map.
% Then $\varphi(\SIDM)=\SIDM$ if and only if one of the following holds:
% \begin{enumerate}
% \item If $n=1$, there exists a scalar $c>0$ such that
% \[
% \varphi(A)=cA
% \qquad\text{for all }A\in M_1(\mathbb R).
% \]

% \item If $n=2$, there exist a permutation matrix $P$ and positive
% diagonal matrices $D_1,D_2$ such that either
% \[
% \varphi(A)=D_1PAP^TD_2
% \qquad\text{or}\qquad
% \varphi(A)=D_1PA^TP^TD_2
% \]
% for all $A\in M_2(\mathbb R)$.

% \item If $n\ge3$, there exist a scalar $c>0$ and a positive monomial
% matrix $S$ such that either
% \[
% \varphi(A)=cSAS^{-1}
% \qquad\text{or}\qquad
% \varphi(A)=cSA^TS^{-1}
% \]
% for all $A\in M_n(\mathbb R)$.
% \end{enumerate}
% \end{theoremA}

\begin{proof}[Proof of Theorem A]
For $n=1$, we have $\SIDM_1=(0,\infty)$. Since every bijective linear
map on $\mathbb R$ has the form $\varphi(A)=cA$ with $c\neq0$, we have
$\varphi(\SIDM_1)=\SIDM_1$ if and only if $c>0$.
For $n=2$, the result follows immediately from Theorem~\ref{imsidm2},
together with the linear preserver characterization of inverse
$M$-matrices.

Now let $n\ge3$. Suppose first that $\varphi(\SIDM)=\SIDM$.
By Lemma~\ref{lem:SIDM_preserver_GL} and the standard linear preserver
theorem for $GL_n(\mathbb R)$, either
$\varphi(A)=PAQ$ or $\varphi(A)=PA^TQ,$ 
for some $P,Q\in GL_n(\mathbb R)$. We treat the first case; the
transpose case is analogous.
\medskip
\noindent
\textbf{(Sufficiency).}
This follows directly from Theorem~\ref{SIDM}, since similarity
commutes with the matrix exponential, permutation and positive
diagonal similarities preserve essential nonnegativity, and
multiplication by $c>0$ corresponds to adding $(\log c)I$ to the
exponent. The transpose case follows similarly.

\medskip
\textbf{(Necessity).}
Assume that
$L(\SIDM)=\SIDM.$ 
We prove that $L(A)=c\,SAS^{-1}$ for some $c>0$ and monomial $S$.

\medskip
\textbf{Step 1: Reduction to nonnegative matrices.}
Let $D=\diag(x_1,\dots,x_n),$ with $x_k>0.$ 
Then $\log D=\diag(\log x_1,\dots,\log x_n)$ is diagonal and hence is essentially nonnegative.
Therefore,
$D=e^{\log D}\in\SIDM.$
Since $L$ preserves \SIDM, we have
\[
PDQ \in \SIDM_n \subseteq M_n(\mathbb R_{\ge 0})
\quad\text{for all } x_1,\dots,x_n>0.
\tag{1}
\]

Fix indices $\alpha,\beta$. The $(\alpha,\beta)$-entry of $PDQ$ is
\[
(PDQ)_{\alpha\beta}
= \sum_{k=1}^n p_{\alpha k}\,x_k\,q_{k\beta}.
\]
This is a linear function of the independent variables $x_k$, and by (1) it is
nonnegative for all choices $x_k>0$. Hence
$p_{\alpha k}\,q_{k\beta} \ge 0$ for all  $\alpha,\beta,k.$ Now fix $k$. Then, all pair-wise products of entries from the $k$th column of $P$ and the
$k$th row of $Q$ are nonnegative.
It follows that every entry in column $j$ of $P$ (and consequently every entry in row $j$ of $Q$) has the same sign or is zero. Furthermore, the signs of the entries from the $j$th column of $P$
and the $j$th row of $Q$ are sign-coherent. 
Hence, there exists a signature matrix $S$ such that the $j$th column of $PS$ and the $j$th row of $SQ$ are both entry-wise nonnegative. Replacing $(P,Q)$ by $(PS,SQ)$ does not affect the map
$A\mapsto PAQ$, and we may therefore assume without loss of generality 

\medskip
\textbf{Step 2: Monomial structure of $P$ and $Q$.}

Since $L(\SIDM)=\SIDM$, the inverse map
$L^{-1}(A)=P^{-1}AQ^{-1}$ 
also preserves \SIDM. Repeating Step 1 gives
$P^{-1}\ge0$ and $Q^{-1}\ge0.$ 
Hence there exist permutation matrices $\Pi,\Gamma$ and positive diagonal matrices $D_1,D_2$ such that
$P=D_1\Pi,$ $Q=\Gamma D_2.$ 

\medskip
\textbf{Step 3: Showing that \(\Gamma=\Pi^{-1}\).}
Since \(I\in \SIDM\) and \(L(\SIDM)=\SIDM\),
we have $PQ=L(I)=D_1\Pi\Gamma D_2\in \SIDM.$
The matrices in \(\SIDM\) have positive main diagonal entries: if \(A=e^B\) with
\(B\) essentially nonnegative, then, for some \(\alpha>0\), $A=e^{-\alpha}e^{B+\alpha I},$ with $B+\alpha I\ge0,$ 
and hence \(A_{ii}>0\). Thus \((PQ)_{ii}>0\) for every \(i\). Since
\(D_1=\diag(a_i)\), \(D_2=\diag(b_i)\), with \(a_i,b_i>0\), we have $(PQ)_{ii}=a_i(\Pi\Gamma)_{ii}b_i.$
Therefore \((\Pi\Gamma)_{ii}=1\) for all \(i\).
Consequently, $L(A)=D_1\Pi A\Pi^{-1}D_2.$

\medskip
\textbf{Step 4: Reduction to a scalar factor.}

%
%Thus, multiplying $P$ on the right and $Q$ on the left by the same signature matrix,
%we may assume without loss of generality that
%\[
%P\ge 0,\qquad Q\ge 0.
%\tag{3}
%\]
%
% (Let
%$S=\diag(\varepsilon_1,\dots,\varepsilon_n)$ with $\varepsilon_k=\pm1$ chosen so
%that the $k$-th column of $PS$ is nonnegative. Then the $k$-th row of $SQ$ is also
%nonnegative. Replacing $(P,Q)$ by $(PS,SQ)$ does not affect the map
%$A\mapsto PAQ$, and we may therefore assume without loss of generality that
%$P\ge 0$ and $Q\ge 0$.)

Now we establish that $D_1D_2$ is a scalar matrix. From Step 3, we have
$L(A)=D_1\Pi A\Pi^{-1}D_2.$ 
Set $S=D_1\Pi.$ 
Then $S$ is monomial and
$L(A)=SAS^{-1}C$, where
$C:=D_1D_2.$ 
Since conjugation by a monomial matrix preserves \SIDM, it suffices to show that if
$AC\in \SIDM$ for all $A\in \SIDM,$ 
where $C=\diag(c_1,\dots,c_n)$ with $c_i>0,$
then $C$ is a scalar matrix.
For the sake of a contradiction, assume that $C$ is not a scalar matrix. Since $n\ge 3$, we may choose indices, relabeled as $1,2,3$, such that
$0<c_1<c_2<c_3.$
Set $N=E_{12}+E_{23}$.
Then $N$ is essential nonnegative matrices and $N^3=0$, so
\[
A(t):=e^{tN}
=
I+tE_{12}+tE_{23}+\frac{t^2}{2}E_{13}
\in \SIDM
\qquad (t\ge 0).
\]
By assumption, $A(t)C\in \SIDM$ $(t\ge 0)$.
Restricting to the principal $3\times 3$ block, we obtain
\[
A(t)C
=
\begin{pmatrix}
c_1 & tc_2 & \frac{t^2}{2}c_3\\
0 & c_2 & tc_3\\
0 & 0 & c_3
\end{pmatrix}.
\]

Applying Lemma~\ref{lem:log-dd} with
$a=c_1, b=c_2,$ and $c=c_3$, yields
$\bigl(\log(A(t)C)\bigr)_{13}<0.$ 
This contradicts the fact that \(\log(A(t)C)\) must be essentially
nonnegative whenever \(A(t)C\in \SIDM\). Hence \(C=cI\) for some
\(c>0\). Since \(C=D_1D_2\), we have \(D_2=cD_1^{-1}\), and therefore
$
L(A)
=
D_1\Pi A\Pi^{-1}D_2
=
c(D_1\Pi)A(D_1\Pi)^{-1}.
$ 
Setting \(S=D_1\Pi\), we obtain
$
L(A)=cSAS^{-1},
$
where \(S\) is monomial. This completes the proof.

Conversely, for $n=1$ the assertion is immediate, since multiplication
by a positive scalar preserves $\SIDM_1=(0,\infty)$. For $n=2$, the
result follows from Theorem~\ref{imsidm2} and the corresponding linear
preserver characterization of inverse $M$-matrices. It remains to
consider $n\ge3$.
Finally, let $n\ge3$ and write
$\varphi(A)=cSAS^{-1}$ or $\varphi(A)=cSA^TS^{-1}$, where $S$ is a
positive monomial matrix and $c>0$. If $A\in\SIDM_n$, then for every
$m\in\mathbb N$ there exists $B_m\ge0$ such that $A=B_m^m$. Since
$A$ is invertible, so is $B_m$. Moreover, respectively,
\[
\varphi(A)=\left(c^{1/m}SB_mS^{-1}\right)^m
\qquad\text{or}\qquad
\varphi(A)=\left(c^{1/m}SB_m^TS^{-1}\right)^m,
\]
and in either case the displayed root is nonnegative. Since
$\varphi(A)$ is invertible, it follows that
$\varphi(A)\in\SIDM_n$. Thus
$\varphi(\SIDM_n)\subseteq\SIDM_n$. Applying the same argument to
$\varphi^{-1}$ gives the reverse inclusion, and hence
$\varphi(\SIDM_n)=\SIDM_n$.
\end{proof}

\section{Infinitely Divisible Matrices: Proof of Theorem B}

% We now consider the linear preserver problem for infinitely divisible nonnegative matrices, which contain \SIDM\ as a distinguished subclass. A key observation is that preservation of \IDM~already forces preservation of the entire cone of entry-wise nonnegative matrices. This allows us to apply our characterization of bijective linear maps preserving entry-wise nonnegativity and, together with the structure imposed by infinite divisibility, obtain a complete characterization of the preservers of \IDM.

% Linear preservers of nonnegative matrices have been studied under various additional assumptions, including spectral and rank-preserving conditions; see, for example, \cite{Minc1974,LiTamTsing1997}. To the best of our knowledge, the characterization needed here, concerning bijective linear maps that preserve entry-wise nonnegativity in both directions, has not been explicitly presented in the literature. We therefore include it for completeness.

We now consider the linear preserver problem for infinitely divisible nonnegative matrices that contain \SIDM\ as a distinguished subclass. A key observation is that preservation of $\IDM$ implies
preservation of the entire cone of entry-wise nonnegative matrices. We first describe the bijective linear operators that preserve entry-wise nonnegativity, and then use the additional structure imposed by infinite divisibility to obtain a complete characterization of the preservers of $\IDM$. 

Linear preservers of nonnegative matrices have been studied under various additional assumptions, including spectral and rank-preserving conditions; see, for example, \cite{Minc1974,LiTamTsing1997}. We include the elementary result needed here for completeness.

\begin{rem}\label{rem:extreme-rays-nonnegative}
{\rm Recall that a ray $\mathbb R_+x$ of a convex cone $C$ is
\textit{extreme} if $x=y+z$, with $y,z\in C$, implies
$y,z\in\mathbb R_+x$. The extreme rays of $M_n(\mathbb R_+)$ are
precisely $\mathbb R_+E_{ij},$ with $1\leq i,j\leq n$ and where \(E_{ij}\) denotes the standard matrix units.
Indeed, a nonnegative matrix with at least two positive entries can
be written as a sum of two distinct and linearly independent nonzero nonnegative matrices. Conversely, if $E_{ij}=X+Y$ with $X,Y\geq0$,
then $X$ and $Y$ must vanish outside the $(i,j)$-entry, and hence both
belong to $\mathbb R_+E_{ij}$.}
\end{rem}

\begin{thm}\label{nonnegativepreserver}
{\rm Let $\Phi:M_n(\mathbb R)\to M_n(\mathbb R)$ be a bijective linear map.
Then
\[
X\geq0 \iff \Phi(X)\geq0
\qquad\text{for all }X\in M_n(\mathbb R)
\]
if and only if there exist positive scalars $\lambda_{ij}$ and a
permutation $\sigma$ of $\{1,\ldots,n^2\}$ such that
$\Phi(E_{ij})=\lambda_{ij}E_{\sigma(i,j)},$ for $1\leq i,j\leq n.$}
\end{thm}

\begin{proof}
Suppose first that
$X\geq0 \iff \Phi(X)\geq0.$ 
Then $\Phi(M_n(\mathbb R_+))=M_n(\mathbb R_+)$, so $\Phi$ is a linear
automorphism of the nonnegative cone. Such an automorphism maps
extreme rays onto extreme rays, as follows immediately by applying
$\Phi^{-1}$ to a decomposition of $\Phi(X)$ into two elements of the
cone. Hence, by Remark~\ref{rem:extreme-rays-nonnegative},
$\Phi(E_{ij})=\lambda_{ij}E_{\sigma(i,j)}$, 
for some $\lambda_{ij}>0$. Since $\Phi$ is injective, differentfferent matrix
units are assigned to different extreme rays. Thus $\sigma$ is injective
and, since $\{1,\ldots,n^2\}$ is finite, it is a permutation.

Conversely, if
$\Phi(E_{ij})=\lambda_{ij}E_{\sigma(i,j)}$ with $\lambda_{ij}>0$
and $\sigma$ a permutation, then linearity immediately gives
$ X\geq0 \iff \Phi(X)\geq0.$  \qedhere
\end{proof}

Next, we show that every bijective linear preserver of $\IDM$ preserves the nonnegative cone.

\begin{lem}\label{IDMpreserver_nonnegative}

{\rm Let
\(
\phi:M_n(\mathbb R)\to M_n(\mathbb R)
\)
be a bijective linear map such that
$
\varphi(\IDM) = \IDM.
$
Then
$
\phi(M_n(\mathbb R_+))=M_n(\mathbb R_+).
$}
\end{lem}

\begin{proof}
We first prove that
$\phi(M_n(\mathbb R_+))\subseteq M_n(\mathbb R_+).$ %Let \(E_{ij}\) denote the standard matrix units.

%\medskip

\noindent
\textbf{Step 1: Diagonal matrix units.}

For each \(i\), we have
$E_{ii}^m=E_{ii}$ for all $m\in\mathbb N.$ 
Thus \(E_{ii}\) is infinitely divisible in \(M_n(\mathbb R_+)\), since for every \(m\in\mathbb N\) we may choose
$(E_{ii})_{(m)}=E_{ii}.$ 
Also \(E_{ii}\ge0\). Hence
$E_{ii}\in F.$ 
By hypothesis,
$\phi(E_{ii})\in F.$ 
In particular,
$\phi(E_{ii})\ge0.$ 

\medskip

\noindent
\textbf{Step 2: Non-diagonal matrix units.}

Let \(i\neq j\). Then
$E_{ij}^2=0.$ 
Fix \(t\ge0\) and \(m\in\mathbb N\). Since \(E_{ij}^2=0\), the binomial formula gives
\[
\left(I+\frac{t}{m}E_{ij}\right)^m
=
I+m\frac{t}{m}E_{ij}
=
I+tE_{ij}.
\]
Moreover,
$I+\frac{t}{m}E_{ij}\ge0.$ 
Therefore, for every \(m\in\mathbb N\), the matrix
$I+\frac{t}{m}E_{ij}$ 
is an entry-wise nonnegative \(m\)th root of \(I+tE_{ij}\). Hence
$I+tE_{ij}\in F.$
By hypothesis,
$\phi(I+tE_{ij})\in F.$ 
Since every element of \(F\) is nonnegative, we get
$\phi(I+tE_{ij})\ge0$ for all $t\ge0$.
By the linearity of \(\phi\),
$\phi(I+tE_{ij})
=
\phi(I)+t\phi(E_{ij}).$ 
Thus
$\phi(I)+t\phi(E_{ij})\ge0$ for all $t\ge0$.

We claim that
$\phi(E_{ij})\ge0.$ 
If some entry of \(\phi(E_{ij})\) is negative, then there exist indices \(k,l\) such that
$(\phi(E_{ij}))_{kl}<0.$
The \((k,l)\)-entry of \(\phi(I)+t\phi(E_{ij})\) is
$
(\phi(I))_{kl}+t(\phi(E_{ij}))_{kl}.
$
Since
$
(\phi(E_{ij}))_{kl}<0,
$
this expression becomes negative for sufficiently large \(t\ge0\). This contradicts
$\phi(I)+t\phi(E_{ij})\ge0$ for all $t\ge0$.
Therefore
$\phi(E_{ij})\ge0$
for all \(i\neq j\).
Combining this with Step 1, we have shown that
$
\phi(E_{ij})\ge0$ for all $i,j.$

\medskip

\noindent
\textbf{Step 3: Nonnegative matrices are mapped to nonnegative matrices.}

Let \(A\in M_n(\mathbb R_+)\) be an entry-wise nonnegative matrix. Then by linearity,
$\phi(A)
=
\sum_{i,j}a_{ij}\phi(E_{ij}),$ and  
$\phi(A)\ge0.$
Thus $\phi(M_n(\mathbb R_+))\subseteq M_n(\mathbb R_+).$

\medskip

\noindent
\textbf{Step 4: Reverse inclusion.}

Since \(\phi\) is bijective and linear, \(\phi^{-1}\) is also linear and bijective.
% We claim that \(\phi^{-1}\) satisfies the same preservation property. Indeed, for any \(B\in M_n(\mathbb R)\), let
% \[
% A=\phi^{-1}(B).
% \]
% Then
% \[
% B=\phi(A).
% \]
% Using the hypothesis
% \[
% A\in F \iff \phi(A)\in F,
% \]
% we get
% \[
% \phi^{-1}(B)\in F \iff B\in F.
% \]
% Thus
% \[
% B\in F \iff \phi^{-1}(B)\in F.
% \]
% So \(\phi^{-1}\) also preserves \(F\) in both directions.
Applying Steps 1--3 to \(\phi^{-1}\), we obtain
$\phi^{-1}(M_n(\mathbb R_+))
\subseteq
M_n(\mathbb R_+).$ 
Now, let \(B\in M_n(\mathbb R_+)\). Then
$\phi^{-1}(B)\in M_n(\mathbb R_+)$ and 
hence $B=\phi(\phi^{-1}(B))
\in
\phi(M_n(\mathbb R_+)).$
Therefore
$ M_n(\mathbb R_+)
\subseteq
\phi(M_n(\mathbb R_+)),$ and by symmetry, 
$\phi(M_n(\mathbb R_+))
=
M_n(\mathbb R_+).$
This completes the proof.
\end{proof}

% We now prove our main result on linear preservers of the class $\IDM$.

% \begin{thm}
% Let $\varphi : M_n(\mathbb{R}) \to M_n(\mathbb{R})$ be a bijective linear map.
% Then 
% \[
% \varphi(\IDM) = \IDM
% \]
% if and only if there exist a permutation matrix $P$, a positive diagonal matrix $D$, and a scalar $c > 0$ such that either
% \[
% \varphi(A) = c\, (PD) A (PD)^{-1} \quad  \text{for all } A,
% \qquad \text{or} \qquad
% \varphi(A) = c\, (PD) A^T (PD)^{-1} \quad \text{for all } A.
% \]   
% \end{thm}
% \begin{proof}
% Suppose
% $\varphi(\IDM)=\IDM.$ 
% Then, by Theorem \ref{IDMpreserver_nonnegative},
% $\varphi(M_n(\mathbb R_+))=M_n(\mathbb R_+).$ 
% Hence, by Theorem \ref{nonnegativepreserver}, we have 
% $\varphi(A)=D_1PAQD_2$
% or
% $\varphi(A)=D_1PA^TQD_2,$ 
% where \(P,Q\) are permutation matrices and \(D_1,D_2\) are positive diagonal matrices.

% Using similar arguments as in Step 3 and Step 4 of Theorem \ref{SIDMpreserver}, we obtain
% $Q=P^T$ and $D_2=cD_1^{-1}$ 
% for some \(c>0\). Therefore,
% $\varphi(A)=cD_1PAP^TD_1^{-1}$ 
% or $\varphi(A)=cD_1PA^TP^TD_1^{-1}.$ 
% It follows that
% $\varphi(A)=c(PD_1)A(PD_1)^{-1}$ 
% or
% $\varphi(A)=c(PD_1)A^T(PD_1)^{-1}.$

% Conversely, let $\varphi(A)=cSAS^{-1}$ 
% or $\varphi(A)=cSA^TS^{-1},$ 
% where \(S\) is a positive monomial matrix and \(c>0\). If $A=B_m^m$ with \(B_m\ge0\), then
% $\left(c^{1/m}SB_mS^{-1}\right)^m=\varphi(A),$ 
% and $c^{1/m}SB_mS^{-1}\ge0.$ 
% Hence \(\varphi(A)\in\IDM\). Applying the same argument to \(\varphi^{-1}\) yields
% $\varphi(\IDM)=\IDM.$ 
% \end{proof}

By Lemma~\ref{IDMpreserver_nonnegative} and Theorem~\ref{nonnegativepreserver}, there exist positive scalars $\lambda_{ij}$ and a permutation $\sigma$ of $\{1,\ldots,n^2\}$ such that $\varphi(E_{ij})=\lambda_{ij}E_{\sigma(i,j)}.$  The preservation of infinite divisibility further restricts the form of $\sigma$ and the scalars $\lambda_{ij}$. We first present two auxiliary facts required for the proof of our main result. 

For nonnegative $2\times2$ matrices, membership in $\SIDM$ has the following simple characterization. \begin{thm}\cite{MST}\label{thm:SIDM2-det}
{\rm Let $A\in M_2(\mathbb R)$ be nonnegative. Then $ A\in\SIDM \Longleftrightarrow \det A>0. $ }\end{thm}

The preceding result characterizes the nonsingular members of $\IDM_2$. The singular case can also be explicitly described, leading to the following complete characterization of $\IDM_2$.

\begin{thm}\label{thm:2x2-IDM-general}
{\rm Let $A\in M_2(\mathbb R)$ be nonnegative. Then
\[
A\in\IDM_2
\quad\Longleftrightarrow\quad
A=0\ \text{or}\ \bigl(\det A\ge0\text{ and }\operatorname{tr}A>0\bigr).
\]}
\end{thm}

\begin{proof}
Suppose $A\in\IDM_2$. Since $A=C^2$ for some $C\ge0$,
$
\det A=(\det C)^2\ge0.
$
If $A\neq0$ and $\operatorname{tr}A=0$, then, since $A\ge0$, its
diagonal entries are zero. Thus
$
A=\begin{pmatrix}0&x\\ y&0\end{pmatrix}
$
with $x,y\ge0$. Since $\det A=-xy\ge0$, we have $xy=0$, and hence
$A^2=0$. But $A=C^2$ for some $C\ge0$, so $C^4=A^2=0$. Thus, $C$ is
nilpotent and, since $C$ is $2\times2$, its nilpotency index is at
most $2$. Hence $C^2=0$, contradicting $C^2=A\neq0$. Therefore
$\operatorname{tr}A>0$.

Conversely, $0\in\IDM_2$. Suppose $A\neq0$, $\det A\ge0$, and
$\operatorname{tr}A>0$. If $\det A>0$, then
Theorem~\ref{thm:SIDM2-det} gives $A\in\SIDM\subseteq\IDM_2$.
If $\det A=0$, put $\tau=\operatorname{tr}A>0$. By
Cayley--Hamilton, $A^2=\tau A$, and hence
$A^m=\tau^{m-1}A$ for every $m\in\mathbb N$. Therefore, for every
$m\in\mathbb N$,
$B_m:=\tau^{1/m-1}A\ge0$
and
$B_m^m=\tau^{1-m}A^m=A.$
Thus $A$ admits a nonnegative $m$th root for every $m\in\mathbb N$,
and hence $A\in\IDM_2$.
\end{proof}

As an immediate consequence, we obtain the following useful criterion for
nonnegative $2\times2$ matrices with positive diagonal entries.

\begin{cor}\label{lem:2x2-IDM}
{\rm Let
$
B=\begin{pmatrix} a&x\\ y&d \end{pmatrix}\ge0,
\quad a,d>0.
$
Then
$
B\in\IDM_2
\Longleftrightarrow
\det B\ge0.
$}
\end{cor}

% \begin{lem}\label{lem:2x2-IDM}
% Let
% \[
% B=\begin{pmatrix} a&x\\ y&d \end{pmatrix}\ge0,
% \qquad a,d>0.
% \]
% Then
% \[
% B\in\IDM_2 \quad\Longleftrightarrow\quad \det B\ge0.
% \]
% \end{lem}

% \begin{proof}
% If $\det B<0$, then $B$ cannot have an even real root.
% Conversely, suppose $\det B\ge0$. Since
% \[
% (\operatorname{tr}B)^2-4\det B=(a-d)^2+4xy\ge0,
% \]
% the eigenvalues of $B$ are real and since $\operatorname{tr}B>0$
% and $\det B\ge0$ these eigenvalues are nonnegative. Write them as
% $\mu_1\ge\mu_2\ge0$.

% If $\mu_1>\mu_2>0$, applying Lagrange interpolation for
% matrix functions, with $0<r<1$,
% \[
% B^r
% =
% \frac{\mu_1^r-\mu_2^r}{\mu_1-\mu_2}\,B
% +
% \frac{\mu_1\mu_2^r-\mu_2\mu_1^r}{\mu_1-\mu_2}\,I.
% \]
% Both coefficients are positive, since $t^r$ is increasing and
% $t^{r-1}$ is decreasing on the interval $(0,\infty)$. Hence $B^r\ge0$.

% %(
% %$
% %f(B)
% %=
% %f(\mu_1)\frac{B-\mu_2I}{\mu_1-\mu_2}
% %+
% %f(\mu_2)\frac{B-\mu_1I}{\mu_2-\mu_1}.
% %
% %Taking $f(t)=t^r$, where $0<r<1$, gives
% %\[
% %B^r
% %=
% %\mu_1^r\frac{B-\mu_2I}{\mu_1-\mu_2}
% %+
% %\mu_2^r\frac{B-\mu_1I}{\mu_2-\mu_1}.
% %\]
% %Collecting the coefficients of $B$ and $I$, we obtain
% %\[
% %B^r=\alpha_rB+\beta_rI,
% %\qquad
% %\alpha_r=\frac{\mu_1^r-\mu_2^r}{\mu_1-\mu_2},
% %\qquad
% %\beta_r=
% %\frac{\mu_1\mu_2^r-\mu_2\mu_1^r}{\mu_1-\mu_2}.
% %)\]

% If $\mu_1>0$ and $\mu_2=0$, then
% $B^r=\mu_1^{r-1}B\ge0$.
% If $\mu_1=\mu_2=\mu$, then 
% $B=\mu I+N,$ with $N\ge0,\qquad N^2=0,$ 
% and therefore
% $
% B^r=\mu^rI+r\mu^{r-1}N\ge0.
% $ 
% Taking $r=1/m$ shows that $B^{1/m}\ge0$ for every $m\in\mathbb N$.
% Hence $B\in\IDM_2$.
% \end{proof}

\begin{rem}\label{rem:3x3-triangular-IDM}
{\rm Let
\[
T=
\begin{pmatrix}
a&x&z\\
0&b&y\\
0&0&c
\end{pmatrix},
\qquad
a,b,c>0,\quad x,y,z\ge0,
\]
and let $f(t)=\log t$. By \cite[Theorem~4.11]{hig},
\[
(\log T)_{12}=x f[a,b],
\qquad
(\log T)_{23}=y f[b,c],
\]
and
$(\log T)_{13}
=
z f[a,c]+xy f[a,b,c].$ 
Since
$f[a,b]>0,$
$f[b,c]>0$, and $f[a,c]>0,$ 
the first two off-diagonal entries of $\log T$ are nonnegative.
Thus $\log T$ is essentially nonnegative if and only if
$z f[a,c]+xy f[a,b,c]\ge0$.
Consequently,
\begin{equation*}\label{eq:3x3-triangular-IDM}
T\in\IDM_3
\quad\Longleftrightarrow\quad
z\ge
-\frac{f[a,b,c]}{f[a,c]}xy.
\end{equation*}
The same criterion applies when this block is embedded as a direct
summand in $M_n(\mathbb R)$ with positive scalar blocks on the
remaining coordinates. This observation will be useful in the proof
of the coefficient structure of $\Phi$, where it allows us to relate
$\lambda_{ij}$, $\lambda_{jk}$, and $\lambda_{ik}$.}
\end{rem}

The next result represents an important step towards describing the bijective linear preservers of the class \IDM.

\begin{pro}\label{prop:IDM-intermediate-form}
{\rm Let $n\ge2$, and let $\Phi:M_n(\mathbb R)\to M_n(\mathbb R)$ be a
bijective linear map satisfying $\Phi(\IDM)=\IDM$. Then there exist a
permutation matrix $P$ and positive diagonal matrices $D_1,D_2$ such
that either $\Phi(A)=D_1PAP^TD_2$ or $\Phi(A)=D_1PA^TP^TD_2$ for all
$A\in M_n(\mathbb R)$.}
\end{pro}

\begin{proof}
By Lemma~\ref{IDMpreserver_nonnegative},
$\Phi(M_n(\mathbb R_+))=M_n(\mathbb R_+).$ 
Since the extreme rays of the nonnegative cone are precisely
$\mathbb R_+E_{ij}$, there exist positive scalars $\lambda_{ij}$ and
a permutation $\sigma$ of the positions of the matrix-unit $n^2$ such that
\begin{equation}\label{eq:Phi-matrix-units}
\Phi(E_{ij})=\lambda_{ij}E_{\sigma(i,j)}.
\end{equation}

Our aim is twofold. First, we determine the combinatorial structure
of $\sigma$ and show that there exists a permutation $p\in S_n$ such
that either
\[
\sigma(i,j)=(p(i),p(j))
\qquad\text{for all }i,j,
\]
or
\[
\sigma(i,j)=(p(j),p(i))
\qquad\text{for all }i,j.
\]
Second, we show that the coefficients factor as
\[
\lambda_{ij}=\alpha_i\beta_j
\qquad\text{for all }i,j
\]
for some positive numbers $\alpha_i,\beta_j$. These two conclusions
yield the required forms.

\medskip
\noindent
\textbf{Step 1: The combinatorial structure of $\sigma$.}

We determine the possible positions of the matrix units under $\Phi$.

\medskip
\noindent
\emph{Step 1.1: Images of the diagonal matrix units.}
For every $i$,
$
E_{ii}^m=E_{ii}
\quad(m\ge1),
$
so $E_{ii}\in\IDM_n$. On the other hand,
$
E_{ij}\notin\IDM_n
\quad(i\neq j).
$
Indeed, if $X^m=E_{ij}$ for some $m\ge n$, then
$X^{2m}=0$, so $X$ is nilpotent. Hence $X^n=0$, and since
$m\ge n$, we obtain $X^m=0$, contradicting $X^m=E_{ij}$.

Thus, the diagonal matrix units are precisely the matrix units
belonging to \IDM. Since $\Phi$ preserves \IDM~ in both
directions, there exist a permutation $p\in S_n$ and positive numbers
$d_i$ such that
\begin{equation}\label{eq:Phi-diagonal-units}
\Phi(E_{ii})=d_iE_{p(i)p(i)}.
\end{equation}

%\medskip
\noindent
\emph{Step 1.2: Possible positions of the off-diagonal matrix units.}
{Fix $i\neq j$. If $n=2$, then, since $\sigma$ is a permutation and the two diagonal positions have already been occupied by the images of $E_{11}$ and $E_{22}$, the two off-diagonal matrix units must be mapped onto the two off-diagonal positions. Hence \[ \sigma(i,j) \in \{(p(i),p(j)),(p(j),p(i))\}. \] Suppose now that $n\ge3$.}
 Since
\[
(E_{ii}+tE_{ij})^2=E_{ii}+tE_{ij},
\qquad
(E_{jj}+tE_{ij})^2=E_{jj}+tE_{ij},
\]
both matrices belong to \IDM for every $t\ge0$.

Since the diagonal positions are precisely the images of the
diagonal matrix units, we have 
$\sigma(i,j)=(a,b),$ 
then $a\neq b$.

Noting that if $r\notin\{a,b\}$ and
$\alpha,\beta>0$, then
\begin{equation}\label{eq:diag-offdiag-not-IDM}
\alpha E_{rr}+\beta E_{ab}\notin\IDM.
\end{equation}
Indeed, the restriction of this matrix to its generalized
$0$-eigenspace is nonzero and nilpotent. If
$Y^m=\alpha E_{rr}+\beta E_{ab}$ 
for some $m\ge n$, then the generalized $0$-eigenspaces of $Y$ and
$Y^m$ coincide. The restriction of $Y$ to this subspace is nilpotent
of index at most $n$, so its $m$th power is zero, contradicting the
nonzero restriction of $Y^m$.

Applying \eqref{eq:diag-offdiag-not-IDM} to the images of
$E_{ii}+tE_{ij}$ and $E_{jj}+tE_{ij}$ shows that both $p(i), p(j) \in \{a,b\}$. Since $p(i)\neq p(j)$,
$ 
\{a,b\}=\{p(i),p(j)\},
$
and therefore
\begin{equation}\label{eq:sigma-two-orientations}
\sigma(i,j)
\in
\{(p(i),p(j)),(p(j),p(i))\}.
\end{equation}

% Moreover, since $\sigma$ is a permutation, the two matrix units
% $E_{ij}$ and $E_{ji}$ must occupy the two distinct positions
% \[
% E_{p(i)p(j)}
% \qquad\text{and}\qquad
% E_{p(j)p(i)}.
% \]
% Thus the orientation is well defined on each unordered pair
% $\{i,j\}$.

Moreover, applying \eqref{eq:sigma-two-orientations} to $(j,i)$ gives
$\sigma(j,i)\in
\{(p(j),p(i)),(p(i),p(j))\}.$ 
Since $\sigma$ is a permutation and $(i,j)\neq(j,i)$, we have
$\sigma(i,j)\neq\sigma(j,i).$ 
Hence exactly one of the following two alternatives holds
$\sigma(i,j)=(p(i),p(j)), \sigma(j,i)=(p(j),p(i))$, or
$\sigma(i,j)=(p(j),p(i)), \sigma(j,i)=(p(i),p(j)).$
We call the first alternative orientation-preserving on the unordered
pair $\{i,j\}$ and the second orientation-reversing. 

\noindent
\emph{Step 1.3: The orientation is global.}
{ If $n=2$, there is only one unordered pair, namely $\{1,2\}$, so the orientation is automatically global. We therefore assume $n\ge3$ for the remainder of this step.}
Suppose, to the contrary, that there exist distinct indices $i,j,k$ such that
the pair $\{i,j\}$ is orientation-preserving while the pair $\{j,k\}$
is orientation-reversing. Thus
\begin{equation}\label{eq:mixed-orientation}
\sigma(i,j)=(p(i),p(j)),
\qquad
\sigma(j,k)=(p(k),p(j)).
\end{equation}

Choose $a_1,\ldots,a_n>0$ such that both
$a_1,\ldots,a_n$ and $d_1a_1,\ldots,d_na_n$ are pairwise distinct,
and set
\begin{equation*}\label{eq:orientation-test}
A=\sum_{r=1}^n a_rE_{rr}+sE_{ij}+tE_{jk},
\qquad s,t>0.
\end{equation*}
After a simultaneous permutation of the coordinates, $A$ is the
direct sum of positive scalar blocks and the $3\times3$ upper
triangular block
\[
\begin{pmatrix}
a_i&s&0\\
0&a_j&t\\
0&0&a_k
\end{pmatrix}.
\]

%For the indices $i,j,k$, the two off-diagonal entries form the directed
%chain
%\[
%i\longrightarrow j\longrightarrow k,
%\]
%and the corresponding matrix units satisfy
%\[
%E_{ij}E_{jk}=E_{ik}\neq0.
%\]

By \cite[Theorem~4.11]{hig}, applied to $f(x)=\log x$, we have 
\begin{equation}\label{eq:logA-second-term}
(\log A)_{ik}=st\,f[a_i,a_j,a_k].
\end{equation}
Since $f[a_i,a_j,a_k]<0$ by Remark~\ref{rem:log-int-rep} and
$s,t>0$, it follows that  $(\log A)_{ik}<0.$ 
 Therefore,
$
A\notin\IDM_n.
$

We now examine the image of $A$. By \eqref{eq:mixed-orientation},
$E_{ij}\longmapsto
\lambda_{ij}E_{p(i)p(j)}$ and  
$E_{jk}\longmapsto
\lambda_{jk}E_{p(k)p(j)}.$
%Thus the chain
%\[
%i\longrightarrow j\longrightarrow k
%\]
%is transformed into
%\[
%p(i)\longrightarrow p(j)\longleftarrow p(k),
%\]
%whose two edges have a common terminal vertex. Put
Set $F=E_{p(i)p(j)}$  and $G=E_{p(k)p(j)}$.
Since $p(i),p(j),p(k)$ are distinct, the multiplication rule for
matrix units gives
\begin{equation}\label{eq:FG-zero}
FG=GF=0.
\end{equation}
%Hence, no product involving both off-diagonal matrix units is nonzero.

Using
$\Phi(E_{rr})=d_rE_{p(r)p(r)},$ 
we obtain
\[ \Phi(A)
=
\sum_{r=1}^n d_ra_rE_{p(r)p(r)}
+s\lambda_{ij}E_{p(i)p(j)}
+t\lambda_{jk}E_{p(k)p(j)}.
\]
After a simultaneous permutation of the coordinates, $\Phi(A)$ is
the direct sum of positive scalar blocks and the $3\times3$ upper
triangular block
\begin{equation*}\label{eq:image-mixed-block}
\begin{pmatrix}
d_i a_i&0&\lambda_{ij}s\\
0&d_k a_k&\lambda_{jk}t\\
0&0&d_j a_j
\end{pmatrix}.
\end{equation*}
By \cite[Theorem~4.11]{hig}, the only nonzero off-diagonal entries
of the logarithm of this block are
\[
\lambda_{ij}s\,f[d_i a_i,d_j a_j]
\qquad\text{and}\qquad
\lambda_{jk}t\,f[d_k a_k,d_j a_j].
\]
Indeed, by \eqref{eq:FG-zero}, no product involving both
off-diagonal matrix units is nonzero, so no mixed second-order term
appears in the divided-difference expansion.

Since $\log x$ is strictly increasing on $(0,\infty)$,
\[
f[u,v]
=
\frac{\log u-\log v}{u-v}>0
\qquad
(u,v>0,\ u\neq v).
\]
The relevant diagonal entries are pairwise distinct by construction,
and
$
\lambda_{ij},\lambda_{jk},s,t>0.
$
Therefore, all off-diagonal entries of $\log\Phi(A)$ are nonnegative.
The remaining blocks are positive scalars, so their logarithms are
scalar blocks. Hence $\log\Phi(A)$ is essentially nonnegative and 
$
\Phi(A)\in\IDM.
$
This contradicts Theorem~\eqref{SIDM}, since
$
\Phi(\IDM)=\IDM
$
which implies $ 
A\in\IDM$ if and only if $\Phi(A)\in\IDM.$ 

The opposite mixed orientation is excluded by the same argument.
Thus, two unordered pairs sharing an index cannot have opposite
orientations.

Let $\{i,j\}$ and $\{k,\ell\}$ be any two unordered pairs. If
they share an index, they have the same orientation. If they are disjoint, consider the intermediate pair
$\{j,k\}$. The pairs $\{i,j\}$ and $\{j,k\}$ share the index $j$,
so they have the same orientation. Similarly, $\{j,k\}$ and
$\{k,\ell\}$ share the index $k$, so they have the same orientation.
Hence $\{i,j\}$ and $\{k,\ell\}$ have the same orientation.

 Consequently, exactly one of
the following alternatives holds:
\begin{equation*}\label{eq:global-orientation-preserving}
\sigma(i,j)=(p(i),p(j))
\qquad\text{for all }i,j,
\end{equation*}
or
\begin{equation*}\label{eq:global-orientation-reversing}
\sigma(i,j)=(p(j),p(i))
\qquad\text{for all }i,j.
\end{equation*}

\noindent
\textbf{Step 2: The coefficient structure.}

Having determined the position of each matrix unit in Step~1, our next aim
is to determine the structure of the positive coefficients
$\lambda_{ij}$. 
Assume first that the global orientation is preserving, so that
$\sigma(i,j)=(p(i),p(j))$, for all $i,j$.
Let $P$ be the permutation matrix corresponding to $p$, and define
$\widetilde{\Phi}(A):=P^T\Phi(A)P$.
Since permutation similarity preserves \IDM,
$\widetilde{\Phi}(\IDM)=\IDM$.

Moreover,
\begin{equation*}\label{eq:normalized-Phi}
\widetilde{\Phi}(E_{ij})=\lambda_{ij}E_{ij},
\qquad
d_i=\lambda_{ii}>0.
\end{equation*}

\noindent
\emph{Step 2.1: Comparing  the scalars $\lambda_{ij}$ and $\lambda_{ji}$.}
We first verify that
$
\lambda_{ij}\lambda_{ji}=d_id_j,
\;(i\neq j).
$
Fix $i\neq j$ and consider
$A(x,y)=I+xE_{ij}+yE_{ji},
\; x,y\ge0.$ 
Up to a simultaneous permutation of coordinates, $A(x,y)$ is the
direct sum of
$
\begin{pmatrix}
1&x\\
y&1
\end{pmatrix}
$
and $I_{n-2}$. Hence, by
Lemma~\ref{lem:2x2-IDM},
$A(x,y)\in\IDM
\quad\Longleftrightarrow\quad
xy\le1.$ 
The corresponding $2\times2$ block of
$\widetilde{\Phi}(A(x,y))$ is given by
\[
\begin{pmatrix}
d_i&\lambda_{ij}x\\
\lambda_{ji}y&d_j
\end{pmatrix},
\]
while the remaining blocks are positive scalars. Again by
Lemma~\ref{lem:2x2-IDM},
\[
\widetilde{\Phi}(A(x,y))\in\IDM
\quad\Longleftrightarrow\quad
\lambda_{ij}\lambda_{ji}xy\le d_id_j.
\]
Since $\widetilde{\Phi}$ preserves \IDM\ in both directions, for
all $x,y\ge0$,
\[
xy\le1
\quad\Longleftrightarrow\quad
xy\le\frac{d_id_j}{\lambda_{ij}\lambda_{ji}}.
\]
Hence
\begin{equation}\label{eq:lambda-pair-relation}
\lambda_{ij}\lambda_{ji}=d_id_j.
\end{equation}

{Now we distinguish the cases $n=2$ and $n\ge3$. If $n=2$, the relation \eqref{eq:lambda-pair-relation} already yields the factorization required. Indeed, set $\alpha_1=1,$ $\beta_1=d_1,$ $\beta_2=\lambda_{12}$, and $\alpha_2=\frac{\lambda_{21}}{d_1}$. Then $\alpha_1\beta_1=\lambda_{11}$, $\alpha_1\beta_2=\lambda_{12},$ $\alpha_2\beta_1=\lambda_{21},$ and by \eqref{eq:lambda-pair-relation}, \[ \alpha_2\beta_2 = \frac{\lambda_{21}\lambda_{12}}{d_1} = d_2 = \lambda_{22}. \] Thus $\lambda_{ij}=\alpha_i\beta_j$ for $1\le i,j\le2$. Hence the desired factorization follows when $n=2$. 

We now assume $n\ge3$. In this case, two additional steps are needed to obtain the factorization.}

\noindent
\emph{Step 2.2: Comparing $\lambda_{ij}$, $\lambda_{jk}$, and
$\lambda_{ik}$.}
Choose $a,b,c>0$ such that both collections
$a,b,c$ and $d_i a,d_j b,d_k c$
are pairwise distinct. Embed the block
\[
\begin{pmatrix}
a&x&z\\
0&b&y\\
0&0&c
\end{pmatrix},
\qquad x,y\ge0,
\]
as a direct summand in $M_n(\mathbb R)$, with positive scalar blocks
on the remaining coordinates. By Remark~\ref{rem:3x3-triangular-IDM},
\begin{equation}\label{eq:T-threshold}
    T\in\IDM_3
\quad\Longleftrightarrow\quad
z\ge-\frac{f[a,b,c]}{f[a,c]}xy.
\end{equation}
The corresponding $3\times3$ block of $\widetilde{\Phi}(T)$ is
\[
\begin{pmatrix}
d_i a&\lambda_{ij}x&\lambda_{ik}z\\
0&d_j b&\lambda_{jk}y\\
0&0&d_k c
\end{pmatrix},
\]
and hence
\begin{equation}\label{eq:image-T-threshold}
\widetilde{\Phi}(T)\in\IDM
\quad\Longleftrightarrow\quad
z\ge
-\frac{\lambda_{ij}\lambda_{jk}}{\lambda_{ik}}
\frac{f[d_i a,d_j b,d_k c]}
     {f[d_i a,d_k c]}xy.
\end{equation}
Since preservation holds in both directions, comparing 
\eqref{eq:T-threshold} and \eqref{eq:image-T-threshold} gives
\begin{equation*}\label{eq:divided-difference-identity}
\frac{f[a,b,c]}{f[a,c]}
=
\frac{\lambda_{ij}\lambda_{jk}}{\lambda_{ik}}
\frac{f[d_i a,d_j b,d_k c]}
     {f[d_i a,d_k c]}.
\end{equation*}
This identity is initially obtained when both triples of diagonal
entries are pairwise distinct. Since divided differences extend
continuously to repeated nodes, this holds for all $a,b,c>0$.

Taking $a=b=c=1$ and using
$
f[1,1]=1,
$ $
f[1,1,1]=-\frac12,
$
we obtain
\begin{equation*}\label{eq:forward-triple}
-\frac12
=
\frac{\lambda_{ij}\lambda_{jk}}{\lambda_{ik}}
\frac{f[d_i,d_j,d_k]}{f[d_i,d_k]}.
\end{equation*}
Applying the same identity to the reversed ordered triple $(k,j,i)$ implies
\[
-\frac12
=
\frac{\lambda_{kj}\lambda_{ji}}{\lambda_{ki}}
\frac{f[d_k,d_j,d_i]}{f[d_k,d_i]}.
\]
Since divided differences are symmetric,
$
f[d_k,d_j,d_i]=f[d_i,d_j,d_k],
$ $
f[d_k,d_i]=f[d_i,d_k],
$
and therefore
\begin{equation}\label{eq:triple-ratio-symmetry}
\frac{\lambda_{ij}\lambda_{jk}}{\lambda_{ik}}
=
\frac{\lambda_{kj}\lambda_{ji}}{\lambda_{ki}}.
\end{equation}
Combining \eqref{eq:triple-ratio-symmetry} with
\eqref{eq:lambda-pair-relation}, we obtain
\[
\begin{aligned}
\left(
\frac{\lambda_{ij}\lambda_{jk}}{\lambda_{ik}}
\right)^2
&=
\frac{
(\lambda_{ij}\lambda_{ji})
(\lambda_{jk}\lambda_{kj})
}{
\lambda_{ik}\lambda_{ki}
}\\
&=
\frac{(d_id_j)(d_jd_k)}{d_id_k}
=d_j^2,
\end{aligned}
\] and since all coefficients are positive,
\begin{equation}\label{eq:lambda-cocycle}
\lambda_{ij}\lambda_{jk}
=
d_j\lambda_{ik}.
\end{equation}

\medskip
\noindent
\emph{Step 2.3: Factorization of the coefficients.}
We now use \eqref{eq:lambda-pair-relation} and
\eqref{eq:lambda-cocycle} to achieve the goal of Step~2, namely,
$\lambda_{ij}=\alpha_i\beta_j$, for all $i,j$.
Fix an index $r$ and define
\[
\alpha_i=
\begin{cases}
\dfrac{\lambda_{ir}}{d_r},&i\neq r,\\[1ex]
1,&i=r,
\end{cases}
\qquad
\beta_j=
\begin{cases}
\lambda_{rj},&j\neq r,\\
d_r,&j=r.
\end{cases}
\]
If $i,j\neq r$ and $i\neq j$, then
\eqref{eq:lambda-cocycle}, applied to $(i,r,j)$, gives
$\lambda_{ir}\lambda_{rj}
=
d_r\lambda_{ij},$
and therefore $\lambda_{ij}
=
\alpha_i\beta_j$.
If $i=r$ or $j=r$, the same identity follows directly from the
definitions. Finally, if $i=j\neq r$, then
\eqref{eq:lambda-pair-relation}, applied to the pair $(i,r)$, gives
$
\alpha_i\beta_i
=
\frac{\lambda_{ir}\lambda_{ri}}{d_r}
=
d_i
=
\lambda_{ii}.
$
Hence
%\begin{equation}\label{eq:lambda-factorization}
$\lambda_{ij}
=
\alpha_i\beta_j$ for all $i,j.$ 
%\end{equation}

Let $\widetilde D_1
=
\operatorname{diag}(\alpha_1,\ldots,\alpha_n)$ and
$\widetilde D_2
=
\operatorname{diag}(\beta_1,\ldots,\beta_n)$.
Then $\widetilde{\Phi}(E_{ij})
=
\widetilde D_1E_{ij}\widetilde D_2$ for all
$i,j$.
Since the matrix units form a basis,
$\widetilde{\Phi}(A)
=
\widetilde D_1A\widetilde D_2$
for all $A\in M_n(\mathbb R).$ 

% Recalling that
% \[
% \widetilde{\Phi}(A)=P^T\Phi(A)P,
% \]
% we obtain
% \[
% \Phi(A)
% =
% P\widetilde D_1A\widetilde D_2P^T.
% \]
% Since permutation conjugation preserves positive diagonal matrices,
% this can be written as
% \[
% \Phi(A)=D_1PAP^TD_2
% \]
% for suitable positive diagonal matrices $D_1,D_2$.

% Finally, suppose that the global orientation is reversing. Since
% transposition preserves \IDM, composition with transposition
% reduces this case to the orientation-preserving case. Hence
% \[
% \Phi(A)=D_1PA^TP^TD_2.
% \]
% Thus, in either case,
% \[
% \Phi(A)=D_1PAP^TD_2
% \qquad\text{or}\qquad
% \Phi(A)=D_1PA^TP^TD_2.
% \]

Recall that
$\widetilde{\Phi}(A)=P^T\Phi(A)P,$ 
which implies
\[
\Phi(A)
=
P\widetilde D_1A\widetilde D_2P^T
=
(P\widetilde D_1P^T)\,PAP^T\,
(P\widetilde D_2P^T).
\]
Set $D_1=P\widetilde D_1P^T$ and 
$D_2=P\widetilde D_2P^T$ as positive diagonal matrices and observe that
$\Phi(A)=D_1PAP^TD_2. $ 
\end{proof} 

It remains to show that, {for $n\geq 3$}, the two positive diagonal matrices in Proposition~\ref{prop:IDM-intermediate-form} are related in the sense that one is a positive scalar multiple of the inverse of the other. This yields a complete characterization of the bijective linear preservers of $\IDM$.

% \begin{theoremB}
% Let
% $\varphi:M_n(\mathbb R)\to M_n(\mathbb R)$ be a bijective linear map.
% Then $\varphi(\IDM)=\IDM$ if and only if one of the following holds:
% \begin{enumerate}
% \item If $n=1$, there exists a scalar $c>0$ such that
% \[
% \varphi(A)=cA
% \qquad\text{for all }A\in M_1(\mathbb R).
% \]

% \item If $n=2$, there exist a permutation matrix $P$ and positive
% diagonal matrices $D_1,D_2$ such that either
% \[
% \varphi(A)=D_1PAP^TD_2
% \qquad\text{or}\qquad
% \varphi(A)=D_1PA^TP^TD_2
% \]
% for all $A\in M_2(\mathbb R)$.

% \item If $n\ge3$, there exist a permutation matrix $P$, a positive
% diagonal matrix $D$, and a scalar $c>0$ such that either
% \[
% \varphi(A)=c(PD)A(PD)^{-1}
% \qquad\text{or}\qquad
% \varphi(A)=c(PD)A^T(PD)^{-1}
% \]
% for all $A\in M_n(\mathbb R)$.
% \end{enumerate}
% \end{theoremB}

\begin{proof}[Proof of Theorem B]
For $n=1$, we have $\IDM_1=[0,\infty)$. Every bijective linear map on
$\mathbb R$ has the form $\varphi(A)=cA$ with $c\neq0$, and
$\varphi(\IDM_1)=\IDM_1$ if and only if $c>0$.

Now let $n\ge2$ and suppose $\varphi(\IDM)=\IDM$. By
Proposition~\ref{prop:IDM-intermediate-form}, there exist a permutation
matrix $P$ and positive diagonal matrices $D_1,D_2$ such that either
\[
\varphi(A)=D_1PAP^TD_2
\qquad\text{or}\qquad
\varphi(A)=D_1PA^TP^TD_2.
\]
If $n=2$, this is precisely the asserted form.
Let $n\ge3$. As in Step~4 of the proof of Theorem~\hyperref[thm:mainA]{A}, preservation of $\IDM$ further yields $D_2=cD_1^{-1}$ for some $c>0$. Hence \[ \varphi(A)=cD_1PAP^TD_1^{-1} \qquad\text{or}\qquad \varphi(A)=cD_1PA^TP^TD_1^{-1}. \] Set $D=P^TD_1P$ and $S=PD$. Then $D$ is positive diagonal, $S$ is positive monomial, and $D_1P=PD=S$. Therefore \[ \varphi(A)=cSAS^{-1} \qquad\text{or}\qquad \varphi(A)=cSA^TS^{-1}. \]

Conversely, the assertion for $n=1$ is immediate. Let $n=2$ and
suppose first that $\varphi(A)=D_1PAP^TD_2$. Then $A\ge0$ if and only
if $\varphi(A)\ge0$, while
$
\det\varphi(A)=\det(D_1)\det(D_2)\det A,
$
so the sign of the determinant is preserved. Moreover, for $A\ge0$,
$\operatorname{tr}A>0$ if and only if at least one diagonal entry is
positive. Permutation similarity permutes the diagonal entries, and
positive diagonal multiplication preserves their zero pattern. Hence
\[
\operatorname{tr}A>0
\quad\Longleftrightarrow\quad
\operatorname{tr}\varphi(A)>0.
\]
Also, $A=0$ if and only if $\varphi(A)=0$. Therefore
Theorem~\ref{thm:2x2-IDM-general} gives
$
A\in\IDM_2
$
if and only if
$
\varphi(A)\in\IDM_2.
$
The transpose case follows similarly.

Finally, let $n\ge3$. The converse follows by the same argument as in
the proof of Theorem~\hyperref[thm:mainA]{A}. Hence
$\varphi(\IDM)=\IDM$.
\end{proof}

\section*{Acknowledgments}
%******************************************************************************************************

S.M.\ Fallat is supported in part by an NSERC Discovery Research Grant, Application No.: RGPIN-2025-05272.
 The work of PIMS Postdoctoral Fellow S.\ Mondal leading to this publication was supported in part by the Pacific Institute for the Mathematical Sciences. 
%\newpage

\end{document}